\documentclass[12pt,a4paper,reqno]{amsart}

\usepackage{anysize}
\marginsize{2.25cm}{2.25cm}{3.5cm}{3.5cm}

\usepackage[utf8]{inputenc}

\usepackage{array}
\usepackage{longtable}
\usepackage{color}

\DeclareMathOperator{\res}{res}
\DeclareMathOperator{\Ind}{Ind}
\DeclareMathOperator{\Sym}{Sym}
\DeclareMathOperator{\ord}{ord}
\DeclareMathOperator{\lcm}{lcm}

\newtheorem{theorem}{Theorem}[section]

\newtheorem{lemma}[theorem]{Lemma}
\newtheorem{proposition}[theorem]{Proposition}

\newtheorem{conjecture}[theorem]{Conjecture}

\begin{document}

\title{About a class of Calabi-Yau differential equations}
\author{Gert Almkvist, Michael Bogner \& Jes\'us Guillera}

\date{}

\begin{abstract}
We explain an experimental method to find CY-type differential equations of order $3$ related to modular functions of genus zero. We introduce a similar class of Calabi-Yau differential equations of order $5$, show several examples and make a conjecture related to some geometric invariants. We finish the paper with a few examples of seven order.
\end{abstract}

\maketitle

\section{Introduction}

Let $M_z$ be a family of Calabi-Yau $n$-folds parameterized by a complex variable $z \in \mathbb{P}^1(\mathbb{C})$. Then periods of the unique holomorphic differential $n$ form on $M_z$ satisfy a linear differential equation which is called differential Calabi-Yau equation. The purely algebraic counterparts of these equations with respect to the characterisation given in \cite[Section 3]{BogCY} are called of CY-type. CY-type differential equations have very nice arithmetical properties which make them of great interest in number theory. For example, it is known that formulas for $1/\pi$ and $1/\pi^2$ are related to CY-type differential equations of orders $3$ and $5$ respectively. For the moment, a satisfactory explanation of the family of formulas for $1/\pi^2$ is not known. In this paper we introduce and study two interesting big families of Calabi-Yau differential equations of orders $3$ and $5$ respectively, and give a few examples of $7$ order.

\section{Duality}

For a differential operator $L=\sum_{i=0}^n a_i \partial^i \in \mathbb{C}(z)[\partial]$, its \textit{dual} is given by 
\[ L^{\vee}=\sum_{i=0}^n (-1)^i \partial^i a_i, \] 
where $\partial^k \alpha$ is the operator
\[
\partial^k \alpha = \sum_{j=0}^k \binom{k}{j} \frac{\partial^j \alpha}{\partial z^j} \partial^{k-j},
\]
which comes from $\partial \alpha = \partial \alpha / \partial z + \alpha \partial,$ by induction.
The properties of duality are \cite{Sch}:
\[
{L^{\vee \vee}}=L, \quad (L_1 L_2)^{\vee}=L_2^{\vee} L_1^{\vee}, \quad (f(z))^{\vee}=f(z). 
\]
By one of their defining properties, CY-type equations $L$ of order $n$ are 
\textit{self-dual}, i.e. there is a function $0\neq\alpha\in\mathbb{C}(z)$ such that $L\alpha=(-1)^n\alpha L^{\vee}$. 
As for any function $f(z)$, the differential equation $f(z)Ly=0$ has the same solutions than $Ly=0$, to be coherent we prove the following lemma:
\begin{lemma}
If $L$ satisfies the CY or self-duality condition with $\alpha=\beta$ then $f(z)L$ satisfies the CY or self-duality condition with $\alpha=\beta f(z)$.
\end{lemma}
\begin{proof}
In the self dual identity
\[
L \beta = (-1)^n \beta L^{\vee},
\]
we introduce $f(z)$ to the left and to the right, and we have
\[
f(z) L \beta f(z) = (-1)^n \beta f(z) L^{\vee} f(z),
\]
Hence
\[
\big(f(z) L \big) \beta f(z) = (-1)^n \beta f(z) \big(f(z) L \big)^{\vee},
\]
which proves that $f(z)L$ satisfies the sef-duality condition with $\alpha=\beta f(z)$.
\end{proof}
For a monic operator $L=\partial^n+\sum_{i=0}^{n-1}a_i\partial^i$, we see that 
\[ 
L \alpha = \alpha \partial^n+(n \alpha'+a_{n-1}\alpha)\partial^{n-1}+\cdots, \quad \alpha L^{\vee}=\alpha \partial^n-\alpha a_{n-1} \partial^{n-1}+\cdots,
\]
and we find that $\alpha'=-2a_{n-1}\alpha/n$. If e.g. $n=4$, then $\alpha'=-2/3 a_3 \alpha$ and identifying the coefficients of $\partial$ of $L\alpha =\alpha L^{\vee}$, we see that
\[
4\alpha'''+3 a_3 \alpha'' + 2a_2 \alpha' + a_1 \alpha = \alpha(-a_3''+2a_2'-a_1).
\]
From $\alpha'=-1/2 \, a_3 \, \alpha$, we can get $\alpha''$ and $\alpha'''$. Replacing these values in the identity above we arrive at the following unique independent relation:
\[ a_1=\frac12 a_2a_3-\frac18 a_3^3+a'_2-\frac34 a_3a'_3-\frac12 a''_3, \] 
see also \cite[Proposition 2.1]{AlZu}. In general for a monic operator of order $n$, identifying the coeficients of $\partial^k$ in the self-dual identity, we get the relations
\[
CY(n,k)=\sum_{j=k}^n \binom{j}{k} \left\{ a_j \frac{\alpha^{(j-k)}}{\alpha}-(-1)^{n-k} a_j^{(j-k)} \right\}=0,
\]
for $k=0,\, 1, \dots n-3$. The $CY(n,k)$ for even $n-k$ follows from $CY(n,k-1)$ by differentiation, so there are $\lfloor (n-1)/2 \rfloor$ relations. Here is a Maple program for computing $CY(n,k)$:
\begin{verbatim}

Alpha:=proc(n,k) local j,G; G(0):=1; 
for j to k do 
G(j):=simplify(expand(diff(G(j-1),z)-2/n*a[n-1](z)*G(j-1)));
od; end; 

CY:=proc(n,k) 
simplify(expand(binomial(n,k)*Alpha(n,n-k)+a[k](z)-(-1)^(n-k)*a[k](z)
+add(binomial(j,k)*(a[j](z)*Alpha(n,j-k)
-(-1)(n-j)*diff(a[j](z),z$(j-k))), j=k+1..n-1))); 
end;

\end{verbatim}

For operators in $\mathbb{C}[z][\vartheta]$ we find the following formula for its dual:
\begin{lemma}
The dual of $\mathcal{D}=\sum_{k=0}^mz^kP_k(\vartheta)$ for $P_k(\vartheta)\in\mathbb{C}[\vartheta]$ is given by
\[\mathcal{D}^{\vee}=\sum_{k=0}^mz^kP_k(-\vartheta-k-1).\] 
\end{lemma}

\begin{proof}
By the definition of the dual,
\[\left(z^k\vartheta\right)^{\vee}=\left(z^{k+1}\partial\right)^{\vee}=-\partial z^{k+1}=-z^{k+1}\partial-(k+1)z^k=z^k(-\vartheta-k-1)\]
holds. As taking duals gives rise to a ring automorphism of $\mathbb{C}[z][\partial]$, we have
\begin{align}
(z^k \vartheta^n)^{\vee} &=(\vartheta^{n-1})^{\vee} (z^k \vartheta)^{\vee} = (-\vartheta-1)^{n-1} z^k(-\vartheta-k-1) \nonumber \\
&= (-\vartheta-1)^{n-2} z^k (-\vartheta-k-1)^2=\cdots=z^k(-\vartheta-k-1)^n. \nonumber
\end{align}
This gives the claim.
\end{proof}

\begin{lemma}\label{monic-nomonic}
If $\mathcal{D}$ is the operator
\[ \mathcal{D}=\sum_{k=0}^m z^k P_k(\vartheta)=\sum_{i=0}^n b_i(z) \partial^i=b_n(z) L, \]
where $L$ is monic; then the following statements are equivalent:
\begin{enumerate}
\item $\mathcal{D}$ satisfies the CY condition with $\alpha=z$.
\item $L$ satisfies the CY condition with $\alpha=z/b_n$.
\end{enumerate}
\end{lemma}
\begin{proof}
We have
\[
\big[ Dz = (-1)^n z D^{\vee} \big] \Leftrightarrow 
\big[ b_n L z = (-1)^n z (b_n L)^{\vee} \big] \Leftrightarrow 
\big[ b_n L z = (-1)^n z L^{\vee} b_n \big].
\]
Hence, introducing $b_n^{-1}$ to the left and the right of both sides, we see that it is equivalent to
\[
L z b_n^{-1} = (-1)^n z b_n^{-1} L^{\vee},
\]
which proves the lemma.
\end{proof}
\noindent If $(a)$ or its equivalent $(b)$ holds, then we say that $\mathcal{D}$ or $L$ satisfy the $YY$ (Yifan-Yang) condition. 

\begin{lemma}\label{CYrel}
For $\mathcal{D}=\sum_{k=0}^mz^kP_k(\vartheta)$, we have $\mathcal{D}z=(-1)^nz\mathcal{D}^{\vee}$ if and only if $P_k(\vartheta)=(-1)^nP_k(-\vartheta-k)$ for all $0\leq k\leq m$. 
\end{lemma}
\begin{proof}
As for every polynomial $P$ we have $P(\vartheta)z=zP(\vartheta+1)$, the lemma follows.
\end{proof}
This lemma provides the simplest way to check that an operator satisfies the Yifan Yang (YY) condition. It implies that for order $3$ and degree $m$ the operator $\mathcal{D}$ satisfying the $YY$ condition, is of the form
\begin{equation}\label{general-operator}
\mathcal{D} = \vartheta^3 + \sum_{k=1}^{m} z^k \left(d_{k,3} \left(\vartheta + \frac{k}{2}\right)^3 + d_{k,1} \left(\vartheta + \frac{k}{2}\right) \right),
\end{equation}
and that for order $4$ and degree $m$ is of the form
\begin{equation}\label{general-operator-4}
\mathcal{D} = \vartheta^4 + \sum_{k=1}^{m} z^k \left(d_{k,4} \left(\vartheta + \frac{k}{2}\right)^4 + d_{k,2} \left(\vartheta + \frac{k}{2}\right)^2+ d_{k,0} \right),
\end{equation}
In general, for odd order $n$, we have
\begin{equation}\label{general-operator-odd}
\mathcal{D} = \vartheta^n + \sum_{k=1}^{m} z^k \sum_{j=0}^{\frac{n-1}{2}} d_{k,j} \left( \vartheta+\frac{k}{2} \right)^{2j+1},
\end{equation}
and if $n$ is even
\begin{equation}\label{general-operator-even}
\mathcal{D} = \vartheta^n + \sum_{k=1}^{m} z^k \sum_{j=0}^{\frac{n}{2}} d_{k,j} \left( \vartheta+\frac{k}{2} \right)^{2j},
\end{equation}
where $d_{k,j} \in \mathbb{C}$. Similar to \cite[Sect. 2]{AlGu2}, rewriting  $\mathcal{D}=\sum_{i=0}^n c_i(z)\vartheta^i\in\mathbb{C}[z][\vartheta]$, we define
\[ P(z):=\exp\left(\frac{2}{n}\int \frac{c_{n-1}(z)}{zc_{n}(z)}dz\right),\ P(0):=1. \] 
For operators of orders $n=3$ or $n=5$, the function $P(z)$ is essential to derive formulas for $1/\pi$ and $1/\pi^2$ \cite{AlGu2}.

\begin{theorem}
If $\mathcal{D}$ be an operator such that $\mathcal{D} z=(-1)^nz\mathcal{D}^{\vee}$, then $P(z)$ is a polynomial. In fact $P(z)=c_n(z)$.
\end{theorem}

\begin{proof}
We see that
\[
\mathcal{D}=\vartheta^n+\sum_{k=1}^m z^k d_{k,n} \left( \vartheta+\frac{k}{2} \right)^n+\mathcal{O}(\vartheta)^{n-2}=\vartheta^n+\sum_{k=1}^m z^k d_{k,n} \vartheta^n+\sum_{k=1}^n z^k d_{k,n} \frac{nk}{2} \vartheta^{n-1}+\cdots.
\]
And by definition of $P(z)$ we get
\[
P(z)=\exp\left(\frac{1}{n}\int \frac{ \sum_{k=1}^m d_{k,n} \, nk z^{k-1}}{1+\sum_{k=1}^m d_{k,n} \, z^k} \right)=1+\sum_{k=1}^m d_{k,n} z^k.
\]
Hence $P(z)=c_n(z)$.
\end{proof}
We define the Yifan-Yang class $\mathcal{YY}_n$ as the class of those operators $\mathcal{D}$ of order $n$ satisfying the $YY$ condition, and also all the other properties of Calabi-Yau operators. Hence, we see that $\mathcal{YY}_n$ is a subclass of $CY_n$.

\section{CY-type differential equations of order $3$}

Let $w_0$, $w_1$ and $w_2$ be the fundamental solutions of $\mathcal{\mathcal{D}}w=0$. As usual we define $q=\exp(w_1/w_0)$, which we can invert to get $z$ as a series of powers of $q$. The function $z(q)$ is the mirror map. We will also need the function $J(q)=1/z(q)$, written as a series of powers of $q$. The mirror map of a Calabi-Yau differential equation has integer coefficients. See \cite[Sect. 2.1]{Zu} for all the necessary conditions.

\subsection{The level}

Let $z_c$ be the smallest positive root of $P(z)$, usually the radius of convergence of the holomorphic solution $w_0$ of $\mathcal{D}w=0$, and suppose that $q_c$ is a root of $dz/dq$ such that $z_c=z(q_c)$. Then $\tau_c=\tau(q_c)$ is given by $q_c=\exp(-\pi \tau_c)$. In this situation we define the level $\ell$ of the differential equation as $4/\tau_c^2$. It is a conjecture that the level is always an integer number.

\subsection{Differential operators in the class $\mathcal{YY}_3$}

Prototypes of CY-type differential equations are those for which $J(q)=1/z(q)$ is one of the Moonshine functions stated in \cite{CoNo}. Most of the projective normal forms of second order equations of this type were found by the experimental methods stated in \cite{LiWi}. In this section, we illustrate a method to produce $\mathcal{YY}_3$ equations out of them. Therefore, we first recall some further operations on differential equations.
Given a differential equation $L\in\mathbb{C}(z)[\partial]$ and an algebraic function $g\in\mathbb{C}(z)^{alg}$, we define their \textit{tensor product} by 
\[ 
L\otimes g=gLg^{-1}.
\]
Moreover, the \textit{symmetric square} of $L$ is given by the monic differential equation $\Sym^2(L)$ of minimal degree w.r.t. $\partial$ whose solutions contain the set $\{y_1y_2\mid L(y_1)=L(y_2)=0\}$.
If especially $L=\partial^2+a_1(z)\partial+a_0(z)$ is irreducible, \cite[Proposition 4.26]{Put} implies that
\begin{equation}\label{symsquare}\Sym^2(L)=\partial^3+a_1(z)+(4a_0(z)+a_1'(z)+2a_1(z)^2)\partial+4a_1(z)a_0(z)+2a'_0(z).\end{equation} 
Moreover, we know by \cite[Corollary 4.7]{BogCY} that the symmetric square of a CY-type operator of order two is  - up to a conjugation with an algebraic function $g$ which is holomorphic at $z=0$  - a CY-type operator of order three. Moreover, both operations do not change $J(q)$. Hence, our method is to start with an operator $L=\partial^2+Q$ taken from \cite{LiWi} and to look for an algebraic fuction $g\in\mathbb{C}(z)^{alg}$ such that $\Sym^2(L)\otimes g$ gives rise to a $\mathcal{YY}_3$ operator.

All of the operators $L=\partial^n+\sum_{i=0}^{n-1}a_i(z)\partial^i \in \mathbb{C}(z)[\partial]$ we are dealing with are \textit{fuchsian}, i.e. there are pairwise distinct $s_1,\dots,s_r\in\mathbb{C}$ such that
\begin{equation}\label{fuchs}
a_i=\frac{p_i}{\prod_{j=1}^{r}(z-s_i)^{n-i}},\ \deg(p_i)=(r-2)(n-i).
\end{equation} 
The points $s_1,\dots,s_r,\infty$ are called the \textit{singularities} of $L$. 
The \textit{indicial equation} of a monic operator $L=\sum_{i=0}^n a_i(z) \partial^i$ at $z=s\in\mathbb{P}^1$ is given by

\begin{equation}\label{Indi}
\Ind_{s}(L)=\sum_{k=0}^n \res_{z=s}\left((z-s)^k a_{n-k}(z)\right)\prod_{j=0}^{n-k-1}(T-j).
\end{equation} 
Its roots $(e_1,\dots,e_n)$ are called the \textit{exponents} of $L$ at $z=s$. Note, that the exponents of $L\otimes g$ at $z=p$ are precisely the exponents of $L$ at $z=s$ shifted by the order of $g$ at $z=s$. For $n=3$, the inditial equation reads
\begin{equation}\label{Indi3}
\Ind_{s}(L)=T(T-1)(T-2)+A_s T(T-1)+B_s T+C_s,
\end{equation} 
where 
\[ 
A_s=\res_{z=s} a_2(z), \quad B_s=\res_{z=s} (z-s)a_1(z), \quad C_s=\res_{z=s} (z-s)^2a_0(z).
\]

To produce operators of $\mathcal{YY}_3$ type, we use the following

\begin{lemma}\label{Methode}
Consider a CY-type differential operator $\mathcal{D}=\sum_{i=0}^mz^iP_i(\vartheta)$ of order three, such that the exponents of $\mathcal{D}$ at $z=0$ are $(0,0,0)$ and the exponents at each finite singularity $s\in\mathbb{C}$ are either $\left(0,\frac{1}{2},1\right)$ or $(-a,0,a)$ for an $a\in(0,1)\setminus\{1/2\}$. Then $\mathcal{D}$ is in the $\mathcal{YY}_3$ class. 
\end{lemma}

\begin{proof}
We have to show that $\mathcal{D}z=z\mathcal{D}^{\vee}$ holds. Rewriting $\mathcal{D}=\sum_{i=0}^3c_i(z)\vartheta^i$, we can assume without loss of generality that $c_3(0)\neq 0$ as the exponents of $\mathcal{D}$ at $z=0$ are $(0,0,0)$. 
Let
\[ L=\frac{1}{b_3(z)} \mathcal{D} = \partial^3+\sum_{i=0}^2 a_i(z)\partial^i \] 
with $a_i(z)=p_i(z)/q_i(z)$, $p_i(z),q_i(z)\in\mathbb{C}[z]$, $\gcd(p_i(z),q_i(z))=1$. Then
\[ z^3 a_3(z)=\lcm(q_0(z),q_1(z),q_2(z))=:l(z) \] 
and it suffices to prove that
\[ P(z)=\exp\left(\frac{2}{3}\int a_2(z)dz\right)=\frac{1}{\alpha}=\frac{b_3(z)}{z}=l(z). \] 
If the exponents of $\mathcal{D}$ at $z=s$ are $\left(-a,0,a\right)$, its indicial equation reads 
\[ \Ind_s(L)=T(T-1)(T-2)+3T(T-1)+(1-a^2)T. \] 
By (\ref{Indi3}) we see that $\res_{z=s}a_2(z)=3$  which implies $\ord_{z=s}(P(z))=2$. On the other hand, (\ref{Indi3}) also implies $\res_{z=s}(z-s)^2 a_0(z)=0$ and hence $\ord_{z=s}(l(z))=2$ by (\ref{fuchs}). 
If the exponents of $\mathcal{D}$ at $z=s$ are $\left(0,\frac{1}{2},1\right)$, we find that
\[ \Ind_s(L)=T(T-1)(T-2)+\frac{3}{2}T(T-1). \] 
Therefore, $\res_{z=s}a_2(z)=\frac{3}{2}$ and $\ord_{z=s}(P(z))=1$. As $L$ is a symmetric square, \ref{symsquare} implies the relation
\[ 
a_0(z)=\frac{1}{3}a_2(z)a_1(z)-\frac{1}{3}a_2(z)a'_2(z)-\frac{2}{27}a^3_2(z)+\frac{1}{2}a'_1(z)-\frac{1}{6}a''_2(z).
\] 
Using $\res_{z=s}a_2(z)=\frac{3}{2}$ and $\res_{z=s}(z-s)a_1(z)=0$ one directly checks that $-\ord_{z=s}a_0(z)=\ord_{z=s}l(z)=1$.
As the roots of both $l(z)$ and $P(z)$ are precisely the finite singularities of $L$ this gives the claim.
\end{proof}

Also note, that if $\mathcal{D}$ is as in Lemma \ref{Methode}, its number of terms equals the number of finite singularities with exponents $\left(0,\frac{1}{2},1\right)$ plus twice the number of finite singularities with exponents $\left(-a,0,a\right)$ for $a\in(0,1)\setminus\{1/2\}$.

\subsection{Example}

All in all, our method to produce $\mathcal{YY}_3$ operators works as follows:

\begin{itemize}
\item Take the symmetric square of an operator contained in \cite{LiWi}.
\item Apply a transformation $z/(cz+1)$ with $c\in\mathbb{Q}$ and conjugate with an algebraic function $g$ such that the number of finite singularities is minimal, $\mathcal{D}$ is as in Lemma \ref{Methode} and the number of finite singularities with exponents $\left(0,\frac{1}{2},1\right)$ is maximal. 
\end{itemize}

We illustrate the method for in the following example: Take the operator in \cite{LiWi} corresponding to the moonshine function $A70$. The finite singularities of its symmetric square $L$ are $z=0$, $z=\infty$ and the roots of the polynomial

\[ p(z):=(3z-1)(2z+1)(z+1)(2z^2-z+1)(4z^3+2z+1). \] 
The exponents of $L$ are $(1,1,1)$ at $z=0$, $(1/2,1,3/2)$ at each of the roots of $p(z)$ and $(-2,-1,0)$ at $z=\infty$. 
The singularity at infinity can be removed by taking a tensor product with a function which has order $2$ at this point. We apply the transformation $\varphi(z)=z/(1-z)$, which changes the singularities to be $z=0$, $z=\infty$ and the roots of
\[ \tilde{p}(z)=(z-1)(z+1)(4z-1)(4z^2-3z+1)(5z^3-z^2-z+1). \] 
Now, the exponents are $(1,1,1)$ at $z=0$, $(-2,-1,0)$ at $z=1$ and $(1/2,1,3/2)$ at all other roots pf $\tilde{p}(z)$ and $z=\infty$. According to Lemma \ref{Methode}, tensoring this operator with
\[ g(z)=\frac{(z-1)^{2}}{z\sqrt{(z+1)(4z-1)(4z^2-3z+1)(5z^3-z^2-z+1)}} \] 
yields the operator
\begin{align*}
&8\,{\vartheta}^{3}-4\,z \left( 2\,\vartheta+1 \right)  \left( 7\,{\vartheta}^{2}+7\,\vartheta+5 \right) +2\,{z}^{2} \left( \vartheta+1 \right)  \left( 56\,{\vartheta}^{2}+112\,\vartheta+97
\right)\\& +{z}^{3} \left( 2\,\vartheta+3 \right)  \left( 8\,{\vartheta}^{2}+24\,\vartheta+71 \right) -10\,{z}^{4} \left( \vartheta+2 \right)  \left( 44\,{\vartheta}^{2}+176\,\vartheta+245 \right)\\& + {z}^{5} \left( 2\,\vartheta+5 \right)  \left( 244\,{\vartheta}^{2}+1220\,\vartheta+1729 \right) + 128\,{z}^{6} \left( \vartheta+3 \right)  \left( {\vartheta}^{2}+6\,\vartheta+10 \right) \\&-320\,{z}^{7} \left( \vartheta+4 \right)  \left( \vartheta+3 \right)  \left( 2\,\vartheta+7 \right)
\end{align*}
which is in the $\mathcal{YY}_3$ class. 

There are two more possibilities to get $\mathcal{YY}_3$-operators of the same degree in $z$. Namely, applying the transformation $\varphi(z)=z/(1+3z)$ and tensoring by
\[g(z)=\frac{(3z+1)^2}{z\sqrt{\left( 5\,z+1 \right)  \left( 4\,z+1 \right)  \left( 8\,{z}^{2}+5\,z+1 \right)  \left( 49\,{z}^{3}+39\,{z}^{2}+11\,z+1 \right)}}\]
yields the operator
\begin{align*}
&8\,{\vartheta}^{3}+4\,z \left( 2\,\vartheta+1 \right)  \left( 25\,{\vartheta}^{2}+25\,\vartheta+19 \right) +14\,{z}^{2} \left( \vartheta+1 \right)  \left( 152\,{\vartheta}^{2}+304\,\vartheta+279
 \right)\\& +{z}^{3} \left( 2\,\vartheta+3 \right)  \left( 6280\,{\vartheta}^{2}+18840\,\vartheta+20311 \right) +2\,{z}^{4} \left( \vartheta+2 \right)  \left( 22340\,{\vartheta}^{2}+89360\,\vartheta+
110667 \right)\\& +{z}^{5} \left( 2\,\vartheta+5 \right)  \left( 48180\,{\vartheta}^{2}+240900\,\vartheta+336617 \right) +8\,{z}^{6} \left( \vartheta+3 \right)  \left( 14668\,{\vartheta}^{2}+88008\,
\vartheta+136865 \right)\\& +31360\,{z}^{7} \left( \vartheta+4 \right)  \left( \vartheta+3 \right)  \left( 2\,\vartheta+7 \right)
\end{align*}
where applying the transformation $\varphi(z)=z/(1+3z)$ and tensoring by
\[ 
g(z)=\frac{(2z-1)^2}{z\sqrt{\left( z-1 \right)  \left( 5\,z-1 \right)  \left( 8\,{z}^{2}-5\,z+1 \right)  \left( 4\,{z}^{3}+4\,{z}^{2}-4\,z+1 \right))}}
\]
yields the operator
\begin{align*}
& 8\,{\vartheta}^{3}-4\,z \left( 2\,\vartheta+1 \right)  \left( 15\,{\vartheta}^{2}+15\,\vartheta+11 \right) +2\,{z}^{2} \left( \vartheta+1 \right)  \left( 364\,{\vartheta}^{2}+728\,\vartheta+653
 \right)\\& -{z}^{3} \left( 2\,\vartheta+3 \right)  \left( 1140\,{\vartheta}^{2}+3420\,\vartheta+3589 \right) +32\,{z}^{4} \left( \vartheta+2 \right)  \left( 115\,{\vartheta}^{2}+460\,\vartheta+547
 \right)\\& -32\,{z}^{5} \left( 2\,\vartheta+5 \right)  \left( 35\,{\vartheta}^{2}+175\,\vartheta+229 \right) -24\,{z}^{6} \left( \vartheta+3 \right)  \left( 44\,{\vartheta}^{2}+264\,\vartheta+425
 \right)\\& +640\,{z}^{7} \left( \vartheta+4 \right)  \left( \vartheta+3 \right)  \left( 2\,\vartheta+7 \right)
\end{align*}

\subsection{Other examples}

Below we show some of the examples that we have found of degrees $3$ to $6$.

\subsubsection*{Degree $3$}

\begin{align}
\mathcal{D}_{11A}=\vartheta^3 & - 2z(2\vartheta+1)(5\vartheta^2+5\vartheta+2)+8z^2(\vartheta+1)(7\vartheta^2+14\vartheta+8) \nonumber \\
& - 22z^3(2\vartheta+3)(\vartheta^2+3\vartheta+2), \label{deq-11A}
\end{align}

\begin{align}
\mathcal{D}_{14A}=\vartheta^3 & - z(2\vartheta+1)(11\vartheta^2+11\vartheta+5)+z^2(\vartheta+1)(121\vartheta^2+242\vartheta+141) \nonumber \\
& - 98z^3(2\vartheta+3)(\vartheta^2+3\vartheta+2), \label{deq-14A}
\end{align}

\begin{align}
\mathcal{D}_{15A}=\vartheta^3 & - z(2\vartheta+1)(7\vartheta^2+7\vartheta+3)+z^2(\vartheta+1)(29\vartheta^2+58\vartheta+33) \nonumber \\
& - 30z^3(2\vartheta+3)(\vartheta^2+3\vartheta+2). \label{deq-15A}
\end{align}

The differential operators $\mathcal{D}_{14A}$ and $\mathcal{D}_{15A}$ are already known but $\mathcal{D}_{11A}$ seems to be new.

\subsubsection*{Degree $4$}

\begin{align}
\mathcal{D}_{13A}=\vartheta^3 & - z(2\vartheta+1)(4\vartheta^2+4\vartheta+1)-z^2(\vartheta+1)(46\vartheta^2+92\vartheta+61) \nonumber \\
& - 4z^3(2\vartheta+3)(8\vartheta^2+24\vartheta+21)-3z^4(\vartheta+2)(9\vartheta^2+36\vartheta+35), \label{deq-13A}
\end{align}

\begin{align}
\mathcal{D}_{17A}=\vartheta^3 & - z(2\vartheta+1)(3\vartheta^2+3\vartheta+1)-z^2(\vartheta+1)(27\vartheta^2+54\vartheta+35) \nonumber \\
& - 2z^3(2\vartheta+3)(7\vartheta^2+21\vartheta+17)-4z^4(\vartheta+2)(4\vartheta^2+16\vartheta+15), \label{deq-17A}
\end{align}

\begin{align}
\mathcal{D}_{19A}=\vartheta^3 & - 3z(2\vartheta+1)(2\vartheta^2+2\vartheta+1)+z^2(\vartheta+1)(22\vartheta^2+44\vartheta+31) \nonumber \\
& + 4z^3(2\vartheta+3)(\vartheta^2+3\vartheta+3)-3z^4(\vartheta+2)(9\vartheta^2+36\vartheta+35), \label{deq-19A}
\end{align}

\begin{align}
\mathcal{D}_{22A}=\vartheta^3 & - 2z(2\vartheta+1)(2\vartheta^2+2\vartheta+1)-4z^2(\vartheta+1)(\vartheta^2+2\vartheta+1) \nonumber \\
& + 2z^3(2\vartheta+3)(9\vartheta^2+27\vartheta+22)-8z^4(\vartheta+2)(4\vartheta^2+16\vartheta+15). \label{deq-22A}
\end{align}

\subsubsection*{Degree $5$}

\begin{align}
\mathcal{D}_{35A}=\vartheta^3 & - z(2\vartheta+1)(5\vartheta^2+5\vartheta+3)+z^2(\vartheta+1)(37\vartheta^2+74\vartheta+61) \nonumber \\
& - 2z^3(2\vartheta+3)(25\vartheta^2+75\vartheta+74)+4z^4(\vartheta+2)(34\vartheta^2+13\vartheta+145) \nonumber \\
& - 70z^5(2\vartheta+5)(\vartheta^2+5\vartheta+6), \label{deq-35A}
\end{align}

\begin{align}
\mathcal{D}_{39A}=\vartheta^3 & - 3z(2\vartheta+1)(3\vartheta^2+3\vartheta+2)+23z^2(\vartheta+1)(5\vartheta^2+10\vartheta+8) \nonumber \\
& - 3z^3(2\vartheta+3)(51\vartheta^2+153\vartheta+140)+253z^4(\vartheta+2)^3 \nonumber \\
& + 6z^5(2\vartheta+5)(9\vartheta^2+45\vartheta+56). \label{deq-39A}
\end{align}

\subsubsection*{Degree $6$}

\begin{align}
\mathcal{D}_{23A}=\vartheta^3 & - z(2\vartheta+1)(\vartheta^2+\vartheta)-z^2(\vartheta+1)(23\vartheta^2+46\vartheta+32) \nonumber \\
& - z^3(2\vartheta+3)(25\vartheta^2+75\vartheta+68)-z^4(\vartheta+2)(58\vartheta^2+232\vartheta+248) \nonumber \\
& - z^5(2\vartheta+5)(16\vartheta^2+80\vartheta+96)-z^6(\vartheta+3)(11\vartheta^2+66\vartheta+88), \label{deq-23A}
\end{align}

\begin{align}
\mathcal{D}_{29A}=\vartheta^3 & - z(2\vartheta+1)(5\vartheta^2+5\vartheta+3)+z^2(\vartheta+1)(23\vartheta^2+46\vartheta+37) \nonumber \\
& - z^3(2\vartheta+3)(5\vartheta^2+15\vartheta+13)-z^4(\vartheta+2)(15\vartheta^2+60\vartheta+68) \nonumber \\
& + 2z^5(2\vartheta+5)(5\vartheta^2+25\vartheta+33)-4z^6(\vartheta+3)(4\vartheta^2+24\vartheta+35), \label{deq-29A}
\end{align}

\begin{align}
\mathcal{D}_{31A}=\vartheta^3 & + 2z(2\vartheta+1)(\vartheta^2+\vartheta+1)-2z^2(\vartheta+1)(7\vartheta^2+14\vartheta+10) \nonumber \\
& - z^3(2\vartheta+3)(47\vartheta^2+141\vartheta+138)-3z^4(\vartheta+2)(53\vartheta^2+212\vartheta+240) \nonumber \\
& - 7z^5(2\vartheta+5)(7\vartheta^2+35\vartheta+46)-3z^6(\vartheta+3)(9\vartheta^2+54\vartheta+80). \label{deq-31A}
\end{align}

\section{Calabi-Yau differential equations of order $5$}

\subsection{Mirror map and Yukawa coupling}

We will denote as $w_0$, $w_1$, $w_2$, $w_3$, $w_4$, the fundamental solutions of $\mathcal{D}w=0$. The mirror map $z(q)$ and the Yukawa coupling of $\mathcal{D}w=0$ are defined as those corresponding to the fourth order pullback of $\mathcal{D}$. However, we can obtain the mirror map and the Yukawa coupling directly from the solutions of the fifth order differential equation. Indeed, from \cite[eqs. 3.19 \& 3.22]{AlGu2} we obtain
\begin{equation}
q=\exp \int \frac{w_0}{w_0 \vartheta w_1 - w_1 \vartheta w_0} \, \frac{1}{z\sqrt{P(z)}} \, dz,
\end{equation}
Expanding the integrand in powers of $z$, integrating term by term, exponentiating and expanding again in powers of $z$, we obtain $q$ as a series of powers of $z$. Then, we can invert it to have $z$ as a series of powers of $q$ (the mirror map). The Yucawa coupling is then given by \cite[eq. 3.22]{AlGu2}, that is
\begin{equation}
K(q)= \left( \frac{q}{z} \, \frac{dz}{dq} \right)^2 \, \frac{1}{w_0 \sqrt{P}}.
\end{equation}
The mirror map of a Calabi-Yau differential equation has integer coefficients. See \cite[Sect. 2.2]{Zu} for all the necessary conditions.

\subsection{Differential operators}

We show some Calabi-Yau differential operators $\mathcal{D}$ which belong to $\mathcal{YY}_5$. We have taken them from \cite{AlEnStZu}:

{\allowdisplaybreaks
\begin{align}
\#32: \qquad \mathcal{D}  &= \vartheta^5-3z(2\vartheta+1)(45\vartheta^4+90\vartheta^3+72\vartheta^2+27\vartheta+4)           \nonumber \\
                          &  \quad -3z^2(\vartheta+1)(9\vartheta^4+36\vartheta^3+53\vartheta^2+34\vartheta+8),            \nonumber \\[1.5ex]
\#60: \qquad \mathcal{D}  &= \vartheta^5-2z(2\vartheta+1)(31\vartheta^4+62\vartheta^3+54\vartheta^2+23\vartheta+4)           \nonumber \\
                          &  \quad +12z^2(\vartheta+1)(144\vartheta^4+576\vartheta^3+839\vartheta^2+526\vartheta+120),    \nonumber \\[1.5ex]
\#189: \qquad \mathcal{D} &= \vartheta^5-2z(2\vartheta+1)(65\vartheta^4+130\vartheta^3+105\vartheta^2+40\vartheta+6)         \nonumber \\
                          &  \quad +16z^2(\vartheta+1)(64\vartheta^4+256\vartheta^3+364\vartheta^2+216\vartheta+45),      \nonumber \\[1.5ex]
\#244: \qquad \mathcal{D} &= \vartheta^5+2z(2\vartheta+1)(26\vartheta^4+52\vartheta^3+44\vartheta^2+18\vartheta+3)           \nonumber \\
                          &  \quad -12z^2(\vartheta+1)(36\vartheta^4+144\vartheta^3+215\vartheta^2+142\vartheta+35),      \nonumber \\[1.5ex]
\#355: \qquad \mathcal{D} &= \vartheta^5-2z(2\vartheta+1)(43\vartheta^4+86\vartheta^3+77\vartheta^2+34\vartheta+6)           \nonumber \\
                          &  \quad +48z^2(\vartheta+1)(144\vartheta^4+576\vartheta^3+824\vartheta^2+496\vartheta+105),    \nonumber \\[1.5ex]
\#356: \qquad \mathcal{D} &= \vartheta^5-2z(2\vartheta+1)(59\vartheta^4+118\vartheta^3+105\vartheta^2+46\vartheta+8)         \nonumber \\
                          &  \quad +384z^2(\vartheta+1)(36\vartheta^4+144\vartheta^3+203\vartheta^2+118\vartheta+24),     \nonumber \\[1.5ex]
\#130: \qquad \mathcal{D} &= \vartheta^5-2z(2\vartheta+1)(14\vartheta^4+28\vartheta^3+28\vartheta^2+14\vartheta+3)           \nonumber \\
                          &  \quad +4z^2(\vartheta+1)(196\vartheta^4+784\vartheta^3+1235\vartheta^2+902\vartheta+255)     \nonumber \\
                          &  \quad -1152z^3(2\vartheta+3)(\vartheta^4+6\vartheta^3+13\vartheta^2+12\vartheta+4),          \nonumber \\[1.5ex]
\#188: \qquad \mathcal{D} &= \vartheta^5-2z(2\vartheta+1)(35\vartheta^4+70\vartheta^3+63\vartheta^2+28\vartheta+5)           \nonumber \\
                          &  \quad +4z^2(\vartheta+1)(1036\vartheta^4+4144\vartheta^3+6061\vartheta^2+3834\vartheta+855)  \nonumber \\
                          &  \quad -1800z^3(2\vartheta+3)(4\vartheta^4+24\vartheta^3+49\vartheta^2+39\vartheta+10),       \nonumber
\end{align}
where the symbol $\#$ stands as a reference of the equation in the Big Tables of \cite{AlEnStZu}.
}
Also very useful is the online database in \cite{Van-St}.

\subsection{New examples}

From some Calabi-Yau fourth order differential operators in \cite[p. 120]{Bog}, we get the following new fifth order Calabi-Yau operators in the class $\mathcal{YY}_5$, by making the Hadamard products indicated below:

\subsubsection*{$\binom{2n}{n}*(4')$:}

\begin{align}
\mathcal{D}   &= \vartheta^5-2^5z(2\vartheta+1)(24\vartheta^4+48\vartheta^3+26\vartheta^2+2\vartheta-1)               \nonumber  \\
              &  \quad -2^{12} z^2(\vartheta+1)(240\vartheta^4+960\vartheta^3+1528\vartheta^2+1136\vartheta+291),  \nonumber \\
              &  \quad -2^{24} z^3(2\vartheta+3)(4\vartheta^4+24\vartheta^3+49\vartheta^2+39\vartheta+10),         \nonumber
\end{align}

\subsubsection*{$\binom{2n}{n}*(5')$:}

\begin{align}
\mathcal{D}   &= \vartheta^5-2^2\,3^2z(2\vartheta+1)(36\vartheta^4+72\vartheta^3+63\vartheta^2+27\vartheta+5)          \nonumber  \\
              &  \quad +2^4\,3^6z^2(\vartheta+1)(144\vartheta^4+576\vartheta^3+904\vartheta^2+656\vartheta+165),    \nonumber \\
              &  \quad -2^{11}\,3^9 z^3(2\vartheta+3)(4\vartheta^4+24\vartheta^3+49\vartheta^2+39\vartheta+10),     \nonumber
\end{align}

\subsubsection*{$\binom{2n}{n}*(6')$:}

\begin{align}
\mathcal{D}   &= \vartheta^5-2^2z(2\vartheta+1)(572\vartheta^4+1144\vartheta^3+795\vartheta^2+223\vartheta+23)           \nonumber \\
              &  \quad -2^4 z^2(\vartheta+1)(9200\vartheta^4+36800\vartheta^3+58184\vartheta^2+42768\vartheta+10863), \nonumber \\
              &  \quad -2^{14}\,3^2 z^3(2\vartheta+3)(4\vartheta^4+24\vartheta^3+49\vartheta^2+39\vartheta+10).       \nonumber
\end{align}

\noindent From some Calabi-Yau fourth order differential operators in \cite{AlBo}, we get other new fifth order CY operators by making the Hadamard products indicated below.

\subsubsection*{$\binom{2n}{n}*(3.5)$:}

\begin{align}
\mathcal{D}   &= \vartheta^5-2^2z(2\vartheta+1)(102\vartheta^4+204\vartheta^3+155\vartheta^2+53\vartheta+7)          \nonumber  \\
              &  \quad +2^4 z^2(\vartheta+1)(1584\vartheta^4+6336\vartheta^3+8768\vartheta^2+4864\vartheta+933),  \nonumber \\
              &  \quad -2^7\,7^2 z^3(2\vartheta+3)(16\vartheta^4+96\vartheta^3+184\vartheta^2+120\vartheta+25),   \nonumber
\end{align}

\subsubsection*{$\binom{2n}{n}*(3.8)$:}

\begin{align}
\mathcal{D}   &= \vartheta^5+2z(2\vartheta+1)(15\vartheta^4+30\vartheta^3+35\vartheta^2+20\vartheta+4)                 \nonumber \\
              &  \quad -2^7 z^2(\vartheta+1)(264\vartheta^4+1056\vartheta^3+1466\vartheta^2+820\vartheta+159),      \nonumber \\
              &  \quad -2^{11}\,7^2 z^3(2\vartheta+3)(16\vartheta^4+96\vartheta^3+184\vartheta^2+120\vartheta+25),  \nonumber
\end{align}

\subsubsection*{$\binom{2n}{n}*(3.9)$:}

\begin{align}
\mathcal{D}   &= \vartheta^5-2z(2\vartheta+1)(113\vartheta^4+226\vartheta^3+173\vartheta^2+60\vartheta+8)            \nonumber  \\
              &  \quad -2^5 z^2(\vartheta+1)(476\vartheta^4+1904\vartheta^3+2629\vartheta^2+1450\vartheta+276),   \nonumber \\
              &  \quad -2^5\,11^2 z^3(2\vartheta+3)(16\vartheta^4+96\vartheta^3+184\vartheta^2+120\vartheta+25),  \nonumber
\end{align}

\subsubsection*{$\binom{2n}{n}*(3.15)$:}

\begin{align}
\mathcal{D}   &= \vartheta^5-2^4z(2\vartheta+1)(21\vartheta^4+42\vartheta^3+30\vartheta^2+9\vartheta+1)              \nonumber  \\
              &  \quad -2^8 z^2(\vartheta+1)(384\vartheta^4+1536\vartheta^3+2132\vartheta^2+1192\vartheta+231),   \nonumber \\
              &  \quad -2^{12}\,5^2 z^3(2\vartheta+3)(16\vartheta^4+96\vartheta^3+184\vartheta^2+120\vartheta+25),  \nonumber
\end{align}

\subsubsection*{$\binom{2n}{n}*\#388$:}

\begin{align}
\mathcal{D}   &= \vartheta^5-2^2z(2\vartheta+1)(582\vartheta^4+1164\vartheta^3+815\vartheta^2+233\vartheta+25)           \nonumber  \\
              &  \quad +2^4 z^2(\vartheta+1)(9264\vartheta^4+37056\vartheta^3+51632\vartheta^2+29152\vartheta+5721),  \nonumber \\
              &  \quad -2^7\,17^2 z^3(2\vartheta+3)(16\vartheta^4+96\vartheta^3+184\vartheta^2+120\vartheta+25).      \nonumber
\end{align}

\begin{center}
\begin{table}
\begin{tabular}{|c||c|c|c|}
\hline &&& \\
$\quad {\rm reference} \quad$ & $\quad \ell_1 \quad $ & $\quad \ell_2 \quad$ & $\quad \ell_3 \quad$  \\[1.5ex]
\hline \hline &&& \\
$\#32$  & $39$  & $117$ & $0$     \\[1.5ex]
$\#37=\binom{4n}{2n}*\#16$ & $24$ & $144$ & $-448$  \\[1.5ex]
$\#39=\binom{2n}{n}*\#16$  & $48$  & $144$ & $-224$  \\[1.5ex]
$\#44=\binom{2n}{n}*\#29$  & $24$  & $96$  & $-164$  \\[1.5ex]
$\#50=\binom{3n}{n}*\#16$  & $24$  & $96$  & $-208$  \\[1.5ex]
$\binom{3n}{n}*\#42$ & $24$ & $108$ & $-231$  \\[1.5ex]
$\binom{4n}{2n}*\#42$ & $16$ & $104$ & $-314$  \\[1.5ex]
$\#60$  & $184$ & $368$ & $-400$  \\[1.5ex]
$\#130$ & $360$ & $360$ & $-240$  \\[1.5ex]
$\#188=\binom{2n}{n}*\#34$ & $120$ & $240$ & $-448$  \\[1.5ex]
$\#189=\binom{2n}{n}*\#28$ & $42$  & $126$ & $-180$  \\[1.5ex]
$\#244$ & $84$  & $168$ & $-168$  \\[1.5ex]
$\#355$ & $132$ & $264$ & $-360$  \\[1.5ex]
$\#356=\binom{2n}{n}*\#205$ & $160$ & $320$ & $-448$  \\[1.5ex]
$\binom{2n}{n}*4'$ & $8$ & $40$ & $-24$  \\[1.5ex]
$\binom{2n}{n}*5'$ & $9$ & $39$ & $-6$  \\[1.5ex]
$\binom{2n}{n}*6'$ & $6$ & $42$ & $-76$  \\[1.5ex]
$\binom{2n}{n}*(3.5)$ & $18$ & $78$ & $-124$  \\[1.5ex]
$\binom{2n}{n}*(3.8)$ & $36$ & $108$ & $-144$  \\[1.5ex]
$\binom{2n}{n}*(3.9)$ & $25$ & $95$ & $-150$  \\[1.5ex]
$\binom{2n}{n}*(3.15)$ & $16$ & $68$ & $-90$  \\[1.5ex]
$\binom{2n}{n}*\#388$ & $6$ & $42$ & $-76$  \\[1.5ex] \hline
\end{tabular}
\vskip 0.5cm
\caption{$\ell$-numbers} \label{ele-num}
\end{table}
\end{center}

\subsection{The $\ell$ numbers}

Let $z_c$ the smallest positive root of $P(z)$ and suppose that $q_c$ is a root of $dz/dq$ such that $K(q_c)=0$ and $z_c=z(q_c)$ (good cases). Then $\tau_c=\tau(q_c)$ is given by
\begin{equation}
\lim_{z \to z_c} \sqrt{P(z)} \, \vartheta^2 w_0 = \frac{1}{\tau_c \, \pi^2},
\end{equation}
an heuristic analogue of \cite[eq. 27]{GuZu} suggested by \cite[Th. 3.2]{AlGu2}. Once we have found the value of $\tau_c$ we can calculate $\alpha_c=\alpha(q_c)$ and $h$ from \cite[eqs. 29 \& 30]{AlGu2}. Notice that in the tables of \cite{AlGu} we used a different notation, namely: $f=\tau_c^2$ and $e=2\alpha_c$. We define the numbers:
\begin{equation}\label{equiv}
\ell_1=\frac{16}{\tau_c^2}, \qquad \ell_2=\frac{16}{\tau_c^2} (6 \alpha_c), \qquad \ell_3=\frac{16}{\tau_c^2} (-h).
\end{equation}
The motivation for the definition of the $\ell$-numbers is to have some characteristics numbers analogous to the geometric invariants $(H^3, \, c_2H, \, c_3)$ of fourth order Calabi-Yau differential equations. First we hint that the $\ell$ numbers are always integers for all the cases we have calculated them. As our formulas diverge when the differential equation has a conifold period, we only deal with fifth order CY differential equations which have no a conifold period. We show some cases in Table \ref{ele-num}.
Has $(\ell_1, \ell_2, \ell_3)$ a geometric interpretation?. It is very curious that, when we calculate the values of $\ell_1$, $\ell_2-\ell_1$ and $2\ell_1+\ell_3$, where $(\ell_1, \ell_2, \ell_3)$ are the $\ell$-numbers of the fifth order operator
\begin{equation}\label{hyp-5}
\vartheta^5 - 32 z (2\vartheta+1)^5,
\end{equation}
we obtain the geometric invariants $H^3$, $c_2H$ and $c_3$ (see \cite{ChYaYu} and \cite{EnSt}) of the fourth order operator \cite[Table 1]{ChYaYu}
\begin{equation}\label{hyp-4}
\vartheta^4 - 16 z (2\vartheta+1)^4,
\end{equation}
and similar coincidences occur for the fourteen hypergeometric cases \cite[Table 1]{ChYaYu}. Thinking about it has lead us to the following conjecture:
\begin{conjecture}\label{conj}
Let $y_0$ be the holomorphic solution of a fourth order CY differential equation of geometric invariants $(H_3, \, c_2 H, \, c_3)$. Let
\[
y_0=\sum_{n=0}^{\infty} A_n z^n, \quad w_0=\sum_{n=0}^{\infty} C_n A_n z^n,
\]
If we choose $C_n$ is such a way that $w_0$ is the holomorphic solution of a fifth order CY differential equation of $\ell$-numbers $(\ell_1, \ell_2, \ell_3)$, then we have the relations shown in Table \ref{relations}.

\begin{center}
\begin{table}
\begin{tabular}{|c||c|c|c|}
\hline &&& \\
$\quad C_n \quad$ & $\quad H^3 \quad $ & $\quad c_2H \quad$ & $\quad c_3 \quad$  \\[1.5ex]
\hline \hline &&& \\
$\binom{2n}{n}$ & $\ell_1$  & $\ell_2-\ell_1$ & $2\ell_1+\ell_3$ \\[1.25ex]
$\binom{4n}{2n}$ & $2\ell_1$ & $2\ell_2-8\ell_1$ & $32\ell_1+2\ell_3$ \\[1.25ex]
$\binom{3n}{n}$ & $\frac43 \ell_1$ & $\frac43 \ell_2-\frac83 \ell_1$ & $8\ell_1+\frac43 \ell_3$ \\[1.25ex]
$\binom{6n}{3n} \binom{3n}{n} \binom{2n}{n}^{-1}$ & $4 \ell_1$ & $4 \ell_2-40 \ell_1$ & $232 \ell_1+4 \ell_3$ \\[1.5ex] \hline
\end{tabular}
\vskip 0.5cm
\caption{relations of Conjecture \ref{conj}} \label{relations}
\end{table}
\end{center}
\end{conjecture}

\noindent Assuming the conjecture and using the values in Table \ref{ele-num}, we get the geometric invariants in Table \ref{geo-inv}. We see that they agree with those in tables of \cite{ChYaYu} and \cite{EnSt} or in Van Straten's on-line Database \cite{Van-St}. This supports our conjecture.

\begin{center}
\begin{table}
\begin{tabular}{|c||c|c|c|}
\hline &&& \\
$\quad {\rm reference} \quad$ & $\quad H^3 \quad $ & $\quad c_2 H \quad$ & $\quad c_3 \quad$  \\[1.5ex]
\hline \hline &&& \\
$\#16$ & $48$ & $96$ & $-128$  \\[1.5ex]
$\#29$  & $24$  & $72$  & $-116$  \\[1.5ex]
$\#34$ & $120$ & $120$ & $-80$  \\[1.5ex]
$\#28$ & $42$  & $84$ & $-96$  \\[1.5ex]
$\#42$ & $32$  & $80$ & $-116$  \\[1.5ex]
$\#205$ & $160$ & $160$ & $-128$  \\[1.5ex]
$4'\sim \#220$ & $8$ & $32$ & $-8$  \\[1.5ex]
$5'\sim \#73$ & $9$ & $30$ & $12$  \\[1.5ex]
$6'\sim\#388$ & $6$ & $36$ & $-64$  \\[1.5ex]
$(3.5)\sim  \#214$ & $18$ & $60$ & $-88$  \\[1.5ex]
$(3.8)\sim \#100 $ & $36$ & $72$ & $-72$  \\[1.5ex]
$(3.9)\sim \#101 $ & $25$ & $70$ & $-100$  \\[1.5ex]
$(3.15)\sim \#328$ & $16$ & $52$ & $-58$  \\[1.5ex] \hline
\end{tabular}
\vskip 0.5cm
\caption{Geometric invariants} \label{geo-inv}
\end{table}
\end{center}

\subsection{Explicit formulae for the hypergeometric cases}

If
\begin{equation}\label{y0w0}
y_0=\sum_{n=0}^{\infty} A_n z^n, \quad w_0=\sum_{n=0}^{\infty} B_n z^n,
\end{equation}
are respectively the holomorphic solutions of a fourth and a fifth order CY differential equation of hypergeometric type, the following expansions hold
\begin{align}
A_x = & 1+\frac16 \frac{c_2H}{H^3} \pi^2 x^2 + \frac{c_3}{H^3} \zeta(3) x^3  \nonumber \\ &+ \left[ \frac{-1}{90} + \frac{1}{18} \left( \frac{c_2 H}{H^3} \right) + \frac{1}{24} \left( \frac{c_2H}{H^3}\right)^2 -\frac{8}{H^3} \right] \pi^4 x^4 + \mathcal{O}(x^5), \label{Ax}
\end{align}
and
\begin{equation}\label{Bx}
B_x=1+\frac16 \frac{\ell_2}{\ell_1}\pi^2x^2+\frac{\ell_3}{\ell_1}\zeta(3)x^3+\left[ \frac{1}{24} \left( \frac{\ell_2}{\ell_1}\right)^2 -\frac{8}{\ell_1} \right] \pi^4 x^4 + \mathcal{O}(x^5).
\end{equation}
Although the analytic continuation of $A_n$ and $B_n$ to $A_x$ and $B_x$ are not known for the non-hypergeometric cases, we believe that these expansions hold in all cases because they explain the relations in Conj. \ref{conj}. We observe that it seems that $(\ell_1, \, \ell_2, \, \ell_3)$ is the analogue of $(H^3, \, c_2 H, \, c_3)$ for the case of fifth order Calabi-Yau differential equations. If we write the hypergeometric cases using the Pochhammer symbol $(s)_n=s(s+1)\cdots(s+n-1)$ instead of binomial numbers:
\[
A_n=\alpha^n \frac{(s_1)_n(1-s_1)_n(s_2)_n(1-s_2)_n}{(1)_n^3}, \qquad B_n=4^n \frac{\left(\frac12\right)_n}{(1)_n}A_n.
\]
where $(s_1,\,s_2)$ are the $14$ possible pairs and $\alpha$ the smallest positive integer such that the numbers $A_n$ are all integers. Then, from (\ref{Ax}) and (\ref{Bx}) we get the explicit formulas
\begin{align}
H^3 &=16 \sin^2 \pi s_1 \sin^2 \pi s_2, \nonumber \\
c_2 H &= H^3 (4+3\cot^2 \pi s_1 + 3 \cot^2 \pi s_2), \nonumber \\
c_3 &=H^3 \left[  \frac43 - \frac{\zeta(3,s_1) + \zeta(3,s_2) + \zeta(3,1-s_1) + \zeta(3,1-s_2)}{3 \, \zeta(3)} \right], \nonumber
\end{align}
and
\begin{align}
\ell_1 &=16 \sin^2 \pi s_1 \sin^2 \pi s_2, \nonumber \\
\ell_2 &=\ell_1 (5+3\cot^2 \pi s_1 + 3 \cot^2 \pi s_2), \nonumber \\
\ell_3 &=-\ell_1 \left[ \frac23+\frac{\zeta(3,s_1)+\zeta(3,s_2)+\zeta(3,1-s_1)+\zeta(3,1-s_2)}{3 \, \zeta(3)} \right], \nonumber
\end{align}
where $\zeta(s,a)$ is the Hurwitz $\zeta$-function.

\section{Maple program}

The following program obtains the polynomial $P(z)$, the mirror and the Yukawa coupling of fifth order Calabi-Yau differential equations. In addition it also finds the $\ell$ numbers of good cases. We conjecture that equivalent cases have the same $\ell$-numbers. So, if a case is bad then we need to find an equivalent good case.

\begin{verbatim}

restart;
with(combinat,stirling2):

V:=proc(n) local j; global L; if n=0 then 1;
else sum(stirling2(n,j)*z^j*Dz^j,j=1..n); fi; end:

Tz:=proc() global pp,V,L,r,w0,w1,w2,w3,w4,qq,m,T,yc,pol,dg,mfinal,dw,dyc;
mfinal:=degree(opd,z); pp:=m->coeff(opd,z,m); print(subs(t=theta,opd));
L:=collect(V(5)+add(add(z^m*coeff(pp(m),t,k)*V(k),m=1..mfinal),k=0..5),Dz):
Order:=61: with(DEtools):
r:=formal_sol(L,[Dz,z],z=0): pol:=z->coeff(opd,t,5); print(P(z)=pol(z));
w0:=r[5]: w1:=r[4]: w2:=r[3]: w3:=r[2]: w4:=r[5]:
dw:=1/(z^2*sqrt(pol(z)))*w0/(w0*diff(w1,z)-w1*diff(w0,z));
qq:=convert(series(exp(int(series(dw,z,51),z)),z,51),polynom):
m:=convert(series(solve(series(qq,z)=q,z),q,51),polynom):
print(J(q)=series(1/m,q,8)); dyc:=1/sqrt(pol(z))*1/w0*(q/z*diff(m,q))^2:
yc:=convert(series(subs(z=m,dyc),q,51),polynom); print(K(q)=series(yc,q,8));
T:=expand(int(expand(int(expand(int((1-yc)/q,q))/q,q))/q,q));
dg:=convert(series(z*diff(z*diff(w0,z),z)*sqrt(pol(z)),z,51),polynom); end:

ele:=proc() global qq0,tt0,vv,alphac,f,h,tauc,zz0,iq0,iv,l1,l2,l3;
Digits:=30; iv:=x->1/x; iq0:=sort(map(iv,-[fsolve(diff(m,q)=0,q)]))[1];
qq0:=-1/iq0; tt0:=evalf(ln(qq0)); if subs(q=qq0,yc)>10^(-6) then return
print("Bad case"); else fi;
tauc:=1/(Pi^2*subs(q=qq0,convert(series(subs(z=m,dg),q,51),polynom)));
f:=evalf(tauc^2);
alphac:=evalf(1/Pi^2*(tt0^2/2-subs(q=qq0,q*diff(T,q)))-tauc);
h:=evalf((-Pi^2*ln(qq0)*alphac+tt0^3/6-subs(q=qq0,T))/Zeta(3));
l1:=convert(16/f,fraction,6); l2:=convert(16*(6*alphac)/f,fraction,6);
l3:=convert(16*(-h)/f,fraction,6);  print(l-numbers=[l1,l2,l3]); end:

calabi:=proc(); Tz(); ele(); end:

\end{verbatim}

\noindent Copy the above code, from a version on-line, and paste it in a Maple session. Suppose that you want to solve the differential equation $\#355$. Then type

\begin{verbatim}

type355:=proc() global opd; print("type355"):
opd:=t^5-2*z*(2t+1)*(43*t^4+86*t^3+77*t^2+34*t+6)
+48*z^2*(t+1)*(144*t^4+576*t^3+824*t^2+496*t+105):
calabi(): end:

type355();

\end{verbatim}

\noindent and execute the program.

\section{Calabi-Yau differential operators of order $7$}

Let
\[
A(n)=\dbinom{2n}{n}^2, \quad B(n)=\dbinom{2n}{n}\dbinom{3n}{n}, \quad
C(n)=\dbinom{2n}{n}\dbinom{4n}{2n}, \quad D(n)=\dbinom{3n}{n}\dbinom{6n}{3n}.
\]
Then using ``Zeilberger" in Maple on (notice that the sum is identically zero by symmetry)
\[
a_n=A(n)\sum_{k=0}^n(n-2k)A(k)^2A(n-k)^2, \quad y_0=\sum_{n^0}^{\infty} a_n z^n,
\]
and similarly for $B,C,D$. For these four cases the function $y_0$ satisfies a CY differential equation of seven order which belongs to the class $\mathcal{YY}_7$. The first remarkable thing about these four differential equations is that the Yukawa coupling can be written 
\[
K(q)=(q\frac{d}{dq})^2(\frac{y_2}{y_0})=1+\sum_{d=1}^{\infty}\frac{d^4n_dq^d}{1-q^d}.
\]
The other remarkable things are
\[
(q\frac{d}{dq})^2(\frac{y_3}{y_0})=K(q)\log (q), \quad  
(q\frac{d}{dq})^2(\frac{y_4}{y_0})=K(q)\frac{\log^{2}(q)}{2}, \quad 
(q\frac{d}{dq})^2(\frac{y_5}{y_0})=\Phi (q)K(q) 
\]
where 
\[ 
\Phi (q)=\dfrac{1}{2}(\dfrac{y_{1}}{y_{0}}\dfrac{y_{2}}{y_{0}}-\dfrac{y_{3}}{y_{0}}) 
\] 
is the Gromov-Witten potential.
The first identity is equivalent to (see [1], p.484)
\[
y_0 y'_3-y'_0y_3=y_1 y'_2-y'_1y_2,
\]
which is equivalent one of the Calabi-Yau conditions for a $4^{th}$ order equation, but is normally not satisfied by $7^{th}$ order equations. Let
\[
T_{jk}=x(y_{j}y'_{k}-y'_{j}y_{k}). 
\]
Then we have 
\[
T_{03}=T_{12}, \quad 2T_{04}=T_{13}, \quad 2T_{05}=T_{23}, \quad T_{16}=T_{34}, \quad, T_{26}=T_{35}, \quad T_{36}=2T_{45}. 
\]
and
\[ 
T_{06}=T_{24}-T_{15}.
\]

\begin{proposition}
\begin{align}
&(a) \quad 2T_{04}=T_{13} \Rightarrow (q\dfrac{d}{dq})^{2}(\dfrac{y_{4}}{y_0})=K(q)\dfrac{\log^{2}(q)}{2} \nonumber 
\\
&(b) \quad 2T_{05}=T_{23} \Rightarrow (q \dfrac{d}{dq})^{2}(\dfrac{y_{5}}{y_{0}})=\Phi (q)K(q). \nonumber
\end{align}
\end{proposition}

\begin{proof}[Proof of (a)] 
We have
\[
y_0^2\frac{d}{dz}(\frac{y_4}{y_0})=\frac{y_1^2}{2}\frac{d}{dz}(\frac{y_3}{y_1}),
\]
that is
\[
\frac{d}{dz}(\frac{y_4}{y_0})=\frac{t^2}{2}\frac{d}{dz}(\frac{y_3}{y_1}),
\]
where $t=y_1/y_0$ . Furthermore $d/dz=dt/dz \cdot d/dt$, implies
\[
\frac{d}{dt}(\frac{y_4}{y_0})=\frac{t^2}{2}\frac{d}{dt}(\frac{y_3}{y_1})=\frac{t^2}{2}\frac{d}{dt}(\frac{1}{t}\frac{y_3}{y_0})=\frac{t}{2}\frac{d}{dt}(\frac{y_3}{y_0})-\frac{1}{2}\frac{y_3}{y_0}. 
\]
It follows
\[
\frac{d^2}{dt^2}(\frac{y_4}{y_0})=\frac{t}{2}\frac{d^2}{dt^2}(\frac{y_3}{y_0})=\frac{t^2}{2}\frac{d^2}{dt^2}(\frac{y_2}{y_0}
) 
\]
by the first identity. 
\end{proof}

\begin{proof}[Proof of (b)]
We have
\[
y_0^2\frac{d}{dz}(\frac{y_5}{y_0})=\frac{y_2^2}{2}\frac{d}{dz}(\frac{y_3}{y_2}). 
\]
As above we get
\[
\frac{d}{dt}(\frac{y_{5}}{y_{0}})=\frac{1}{2}(\frac{y_{2}}{y_{0}})^{2}\frac{d}{dt}(\frac{y_{3}}{y_{2}})=\frac{1}{2}(\frac{y_{2}}{y_{0}})^{2}\frac{d}{dt}(\frac{y_3/y_0}{y_2/y_0})=\frac{1}{2}\left\{\frac{y_{2}}{y_{0}}\frac{d}{dt}(\frac{y_{3}}{y_{0}})-\frac{y_{3}}{y_{0}}\frac{d}{dt}(\frac{y_{2}}{y_{0}})\right\},
\]
which implies
\[
\frac{d^{2}}{dt^{2}}(\frac{y_{5}}{y_{0}})=\frac{1}{2}\left\{ \frac{y_{2}}{y_{0}}\frac{d^{2}}{dt^{2}}(\frac{y_{3}}{y_{0}})-\frac{y_{3}}{y_{0}}\frac{d^{2}}{dt^{2}}(\frac{y_{2}}{y_{0}})\right\} 
=\frac{1}{2}\left\{ \frac{y_{2}}{y_{0}}\frac{y_{1}}{y_{0}}-\frac{y_{3}}{y_{0}}\right\} \frac{d^{2}}{dt^{2}}(\frac{y_{2}}{y_{0}})=\Phi (q)K(q).
\]
\end{proof}

\subsection{Case A}
In this case we find the $7^{th}$ order differential operator
\[
\vartheta^7-128z(2\vartheta+1)^3 (8\vartheta^4+16\vartheta ^3+20\vartheta^2+12\vartheta+3)+2^{20}z^2(\vartheta+1)^3(2\vartheta+1)^2(2\vartheta+3)^2.
\]
To get an explicit formula for $a_n$ one differentiates
\[
a_n=A(n)\sum_{k=0}^{n}\left\{1+k\frac{d}{dk}\right\} A(k)^{2}A(n-k)^2,
\]
obtaining
\[
a_n=\dbinom{2n}{n}^2\sum_{k=0}^n\dbinom{2k}{k}^4\dbinom{2n-2k}{n-k}^{4} \left\{ 1+8k(H_{2k}-H_{k}-H_{2n-2k}+H_{n-k}) \right\}, 
\]
where $H_{0}=0$ and $H_n=1+1/2+\cdots+1/n$ is the harmonic number for $n\geq 1$.

\subsection{Case B}
In this case
\begin{align}
a_n &=\dbinom{2n}{n}\dbinom{3n}{n}\sum_{k=0}^{n}\dbinom{2k}{k}^2\dbinom{3k}{k}^2\dbinom{2n-2k}{n-k}^{2}\dbinom{3n-3k}{n-k}^{2} \nonumber \\
&\times \left\{1+6k(H_{3k}-H_{k}-H_{3n-3k}+H_{n-k})\right\}. \nonumber
\end{align}
The operator corresponding to the differential equation is
\begin{align}
\vartheta^7 &-27z(2\vartheta+1)(3\vartheta+1)(3\vartheta+2)(81\vartheta^4+162\vartheta^3+198\vartheta^2+117\vartheta+28) \nonumber \\
&+3^{12}z^2(\vartheta+1)(3\vartheta+1)(3\vartheta+2)^2(3\vartheta+4)^2(3\vartheta+5). \nonumber
\end{align}
And the instanton numbers
\[
n_1=1485, \, n_2=\frac{9853515}{8}, \, n_3=2555194005, \, n_4=8549298943740, \dots
\]
\subsection{Case C}
Here
\begin{align}
a_n &=\dbinom{2n}{n}\dbinom{4n}{2n}\sum_{k=0}^n\dbinom{2k}{k}^2\dbinom{4k}{2k}^2\dbinom{2n-2k}{n-k}^2\dbinom{4n-4k}{2n-2k}^{2} \nonumber \\ 
& \times \left\{ 1+4k(2H_{4k}-H_{k}-H_{2k}+H_{2n-2k}-2H_{4n-4k}+H_{n-k}) \right\}. \nonumber
\end{align}
And the differential operator is
\begin{align}
\vartheta^7 &-128z(2\vartheta+1)(4\vartheta+1)(4\vartheta+3)(128\vartheta^4+256\vartheta^3+304\vartheta^2+176\vartheta+39) \nonumber \\ 
&+2^{26}z^{2}(\vartheta+1)(2\vartheta+1)(2\vartheta+3)(4\vartheta+1)(4\vartheta+3)(4\vartheta+5)(4\vartheta+7). \nonumber
\end{align}
The instanton numbers are
\[
n_1=29400, \, n_2=277414560, \, n_3=7671739956480, \, n_4=346114703998148120, \dots
\]

\subsection{Case D}
For this case, we have
\begin{align}
a_n &=\dbinom{3n}{n}\dbinom{6n}{3n}\sum_{k=0}^n\dbinom{3k}{k}^{2}\dbinom{6k}{3k}^2\dbinom{3n-3k}{n-k}^2\dbinom{6n-6k}{3n-3k}^2 \nonumber \\ 
& \times \left\{ 1+6k(2H_{6k}-H_{k}-H_{3k}+H_{3n-3k}-2H_{6n-6k}+H_{n-k}) \right\}. \nonumber
\end{align}
The corresponding Calabi-Yau operator is
\begin{align}
\vartheta^7 & -2^7 3^3 z(2\vartheta+1)(6\vartheta+1)(6\vartheta+5)(648\vartheta^4+1296\vartheta^3+1476\vartheta^2+828\vartheta+155) \nonumber \\ &+ 2^{20}3^{12}z^2(\vartheta+1)(3\vartheta+1)(3\vartheta+2)(6\vartheta+1)(6\vartheta+5)(6\vartheta+7)(6\vartheta+11), \nonumber
\end{align}
with instanton numbers
\begin{align}
& n_1=17342208, \, n_2=42976872163296, \, n_3=380850322188446486784, \nonumber \\
& n_4=5581133974953140362085043072, \dots \nonumber
\end{align}
The four differential equations $A$, $B$, $C$ and $D$ were first found by Dettweiler and Reiter
using a different method. They are particular cases up to rescaling $z \to \lambda z$ of the operator $P_1$ in \cite[p. 15]{DeRe} for $c=d=0$, by putting $a=1/2,1/3,1/3,1/6$ respectively.

\subsection{Transformation of case A}
We make the transformation
\[
Y_0(z)=\frac{2^{1/4}}{(1-512z+\sqrt{1-1024z})^{1/4}}\sum_{n=0}^{\infty}\dbinom{2n}{n}^4 \left(-\frac{2z}{1-512z+\sqrt{1-1024z}}\right)^n.
\]
Then $Y_0(z)$ satisfies
\[
\vartheta^4-16z(4\vartheta+1)(32\vartheta^3+40\vartheta^2+28\vartheta+7)+2^{12}z^2(4\vartheta+1)(4\vartheta+3)^2(4\vartheta+5), 
\]
which is $\# 31$. The wronskian $z(Y_{0}Y_{1}^{'}-Y_{0}^{'}Y_{1})$ satisfies 
\begin{align}
\vartheta^5 &-32z(2\vartheta+1)(48\vartheta^4+96\vartheta^3+124\vartheta^2+76\vartheta+21)+2^{18}z^2(\vartheta+1)^3(12\vartheta^2+24\vartheta+23) \nonumber \\
&-2^{29}z^3(\vartheta+1)^2(\vartheta+2)^2(2\vartheta+3). \nonumber
\end{align}
Taking the Hadamard product with $\dbinom{2n}{n}^{2}$, we obtain
\begin{align}
\vartheta^7 &-128z(2\vartheta+1)^3(48\vartheta^4+96\vartheta^3+124\vartheta^2+76\vartheta+21) \nonumber \\
&+2^{22}x^2(\vartheta+2) (2\vartheta+1)^2(2\vartheta+3)^2(12\vartheta^2+24\vartheta+23) \nonumber \\
&-2^{35}z^3(2\vartheta+1)^2(2\vartheta+3)^3(2\vartheta+5)^2, \nonumber
\end{align}
with
\[
n_1=768, \, n_2=-136800, \, n_3=35597568, \, n_4=-5313408000, \dots. 
\]
Trying to make a similar transformation of cases $B,C,D$ does not work. We finally mention that S. Reiter sent us ten other examples of operators in $\mathcal{YY}_7$ (unpublished).

\section{Conclusion}
Calabi-Yau differential equations are important by themselves and their interest in relation with the series for $1/\pi$ and $1/\pi^2$ of Ramanujan-Sato type was explained in \cite{AlGu2}. A recent interesting idea related to this relation is in \cite{GuZu}, and in some cases it leads to completely modular-free proofs of Ramanujan-type series for $1/\pi$. Finally we want to observe that we have noticed an error in \cite[Table-$\delta$]{AlGu}: We suspected that the numbers $e$ and $f$ of that table were wrong because they lead to non-integers values of the $\ell$-numbers. We now know that we were right, and that the error was due to the fact that for the case $A*\delta$, we confused a ``divergent" series with the limit case at the radius of convergence. In fact we have discovered that instead of a limit case, it was the ``divergent" series associated to the supercongruences
\[
\sum_{n=0}^{p-1}\binom{2n}{n}^2\sum_{k=0}^{n}(-1)^{k}3^{n-3k}\binom{n}{3k}\binom{n+k}{n}\frac{(3k)!}{k!^{3}}(39+172n+204n^{2})\frac{1}{(-64)^{n}} \equiv 39p^2 \pmod{p^3},
\]
where $p>3$ is a prime number. Although we are unable to determine the true values of $\ell_1$, $\ell_2$, $\ell_3$ of this bad case, we guessed that $h=-\ell_3/\ell_1=9/2$ which is all we need to find the formulas for $1/\pi^2$ (see \cite{AlGu} and \cite{AlGu2}).

\end{document}